\documentclass[11pt]{amsart}

\usepackage{amssymb, amsmath, mathrsfs}
\usepackage{amsthm}
\usepackage[margin=1.25 in]{geometry}
\usepackage{parskip}
\usepackage{graphicx}
\usepackage{epsfig}
\usepackage{color}
\usepackage{enumitem}
\usepackage{caption}
\usepackage{subcaption}
\usepackage{hyperref}

\setenumerate[1]{label=(\thesection.\arabic*)}

\numberwithin{equation}{section}
\numberwithin{figure}{section}

\usepackage{verbatim}
\newtheorem{theorem}{Theorem}[section]
\newtheorem*{theorem*}{Theorem}          %theorem without number

\newtheorem*{prop*}{Proposition}  % proposition without number
\newtheorem{coro}{Corollary}[section]
\newtheorem{lemma}[theorem]{Lemma}

\theoremstyle{definition}

\theoremstyle{remark}

\def\St{\ensuremath{\mathcal{S}}}

\newcommand{\pa}[1]{\left(#1\right)}
\newcommand{\cpa}[1]{\left\{#1\right\}}

\newcommand{\tn}[1]{\textnormal{#1}}

\newcommand{\card}[1]{\left| #1 \right|}
\newcommand{\mcg}[1]{\tn{MCG}\pa{#1}}
\newcommand{\Int}[1]{\tn{Int} \, #1}

\begin{document}

\title{Prime theta-curves on minimal genus Seifert surfaces}

\author{Jack S. Calcut and Jamie Phillips-Freedman}

\makeatletter
\@namedef{subjclassname@2020}{%
  \textup{2020} Mathematics Subject Classification}
\makeatother

\keywords{Theta-curve, prime, Seifert surface, minimal genus.}
\subjclass[2020]{57K10; Secondary 57M15.}
\date{\today}

\begin{abstract}
     We prove that each prime knot union an essential arc on a minimal genus Seifert surface is a prime theta-curve.
\end{abstract}
\maketitle
\section{Introduction}
The classification of prime knots is a fundamental challenge in knot theory. Knotted graph theory is the study of embeddings of graphs in 3-manifolds.
\begin{figure}[htb!]
    \centering
    \includegraphics{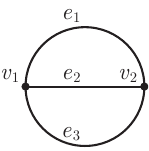}
    \caption{The theta-graph.}
    \label{fig:theta-graph}
\end{figure}
A \textbf{theta-curve} is an embedding of the \textbf{theta-graph}---see Figure \ref{fig:theta-graph}---in the 3-sphere. Theta-curves arise naturally in various contexts and have been well-studied---see Baker, Buck, Moore, O'Donnol and Taylor~\cite{Baker}, the first author and Nieman~\cite{calcut2025}, and references therein.  

Theta-curves have two natural connected sum operations: add a knot to an edge of a theta-curve and combine two theta-curves at a vertex. Those two operations yield the notion of a prime theta-curve. The classification of prime theta-curves is a fundamental problem in knotted graph theory. Prime knots have been tabulated up to 20 crossings by Thistlethwaite~\cite{thistlethwaite} and work of others. The current tabulation of theta-curves is more modest at seven crossings---see Moriuchi~\cite{Moriuchi}. A natural approach to construct a prime theta-curve is to add an arc to a prime knot. That was studied for knots and arcs on a torus by the first author and Nieman~\cite{calcut2025}. Here, we use arcs on minimal genus Seifert surfaces. The main purpose of the present paper is to prove the following. 

\begin{theorem}\label{thm:main} Let $K\subset S^3$ be a prime knot and $\Sigma\subset S^3$ a minimal genus Seifert surface for $K$. Then, $K$ union any arc $\alpha$ neatly embedded and essential in $\Sigma$ is a prime theta-curve.   
\end{theorem}

Theorem \ref{thm:main} yields a broad collection of prime theta-curves: prime knots abound,  every knot has a minimal genus Seifert surface, and each Seifert surface $\Sigma$ of positive genus contains a neatly embedded essential arc. For example, any non-separating such arc is essential as in Figure \ref{fig:thm_example}. 
  \begin{figure}[h]
         \centering
         \includegraphics{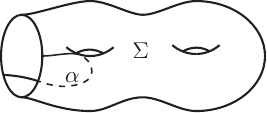}
         \caption{A neatly embedded nonseparating arc $\alpha$.}
         \label{fig:thm_example}
     \end{figure}    
Essential arcs on Seifert surfaces are discussed further in Section \ref{sec:applications}.
Theorem~\ref{thm:main} may be viewed as a necessary condition for a Seifert surface of a given knot to be of minimal genus.
That raises the question of whether the converse holds---see the end of Section~\ref{sec:applications}.

This paper is organized as follows. Section~\ref{sec:def} gives definitions and reviews properties of the genera of surfaces. Section~\ref{sec:proof} proves Theorem~\ref{thm:main}. Section~\ref{sec:applications} explores the infinite collection of prime theta-curves provided by Theorem~\ref{thm:main} and poses a few questions.

\section{Definitions and Background}\label{sec:def}

We work in the piecewise linear category. All results hold in the topological (locally flat) and differentiable categories since those three categories are equivalent in our relevant dimensions $\leq3$.
Fix an orientation of the 3-sphere $S^3$.
Knots and graphs are unlabeled and unoriented.
Two knots or graphs in $S^3$ are \textbf{equivalent} provided an isotopy of $S^3$
carries one to the other set-wise.
We adopt the usual convention that the unknot is not prime. A theta-curve is \textbf{trivial} or \textbf{unknotted} provided it lies on a 2-sphere embedded in $S^3$. A \textbf{$k$-prong} is formed by gluing together one boundary point from each of $k\geq2$ arcs. A \textbf{ball-prong pair} is a pair $(B,P)$ where $P$ is a $k$-prong \textbf{neatly embedded}\footnote{A manifold or $k$-prong $M$ embedded in a manifold $N$ is \emph{neatly embedded}
provided $M$ is a closed subspace of $N$, $\partial M=M \cap \partial N$, and $M$ meets $\partial N$ transversely.} in $B$. A ball-prong pair is unknotted if it is homeomorphic to the unit ball with $k$ radial arcs. Note that a \textbf{ball-arc pair} is the special case $k=2$. The following lemma is useful for identifying unknotted ball-prong pairs. See \cite[Lem. 2.1]{calcut2025} for a proof. 

\begin{lemma}\label{lem:ballpairs}
    Let $(B,P)$ be a ball-prong pair. Then $(B, P)$ is unknotted if and only if there exists a 2-disk $D$ neatly embedded in $B$ and containing $P$.
\end{lemma}

A \textbf{2-connected sum} of a theta-curve $\theta$ and a knot $K$, denoted $\theta\#_2 K$, is formed analogously to the connected sum of two knots. Assume $\theta$ and $K$ lie in separate copies of $S^3$. Let $(B_1,\alpha_1)$ be an unknotted ball-arc pair where $\alpha_1\subset \text{Int}~e_i$ for some arc $e_i$ of $\theta$. Let $(B_2,\alpha_2)$ be an unknotted ball-arc pair for some arc $\alpha_2\subset K$. First, delete the interiors of $B_1$ and $B_2$ from their respective copy of $S^3$. Then, glue together the resulting balls using an orientation-reversing homeomorphism of their boundary 2-spheres matching up the boundary points of $\alpha_1$ and $\alpha_2$. A result is $\theta\#_2 K$.

A \textbf{3-connected sum} of two theta-curves $\theta_1$ and $\theta_2$, denoted $\theta_1\#_3\theta_2$, is formed as follows. Assume that $\theta_1$ and $\theta_2$ lie in separate copies of $S^3$. Let $(B_1,P_1)$ and $(B_2,P_2)$ be unknotted ball-prong pairs where $P_i\subset \theta_i$ is a 3-prong. First, delete the interiors of $B_1$ and $B_2$ from their respective copy of $S^3$. Then, glue together the resulting ball-prong pairs using an orientation-reversing homeomorphism of their boundary 2-spheres matching up the boundary points of $P_1$ and $P_2$. A result is $\theta_1\#_3 \theta_2$.

Recall that the connected sum of two \emph{unoriented} knots may produce up to two knots---see Budney's~\cite{budney} MathOverflow answer using the $8_{17}$ knot.
Similarly, without additional data, the 2-connected sum may produce up to six theta-curves,
and the 3-connected sum may produce up to 24 theta-curves.
Wolcott~\cite[$\S$4]{wolcott}
proved that the 3-connected sum is well-defined once one specifies both the vertices in $\theta_1$ and $\theta_2$
at which to sum and the matching of their edges.

A theta-curve $\theta$ is \textbf{prime} provided the following three conditions are satisfied: 
\begin{enumerate}
    \item[(i)] $\theta$ is knotted,
    \item[(ii)] $\theta\neq\theta'\#_2 K$ for some possibly trivial theta-curve $\theta'$ and some nontrivial knot $K$, and
    \item[(iii)] $\theta\neq\theta_1\#_3\theta_2$ for some nontrivial theta-curves $\theta_1$ and $\theta_2$.
\end{enumerate}
A theta-curve is \textbf{composite} provided it is not prime and is nontrivial. 

A simple closed curve $C$ in the interior of a surface $\Sigma$ is \textbf{inessential} if it bounds a disk in $\Sigma$, or $C$ union a boundary component of $\Sigma$ bound an annulus in $\Sigma$.
Otherwise, $C$ is \textbf{essential}.
An arc $\alpha$ neatly embedded in $\Sigma$ is \textbf{inessential} if $\alpha$ union an arc in the boundary of $\Sigma$ bound a disk in $\Sigma$, and is \textbf{essential} otherwise. An arc $\alpha$ on a connected surface $\Sigma$ is \textbf{separating} provided $\Sigma-\alpha$ is disconnected.

Let $\Sigma$ be a compact, orientable, connected surface.
The \textbf{genus} of $\Sigma$, denoted $g(\Sigma)$, is defined by:
\begin{equation}\label{eq:gen}
    g(\Sigma)=\frac{2-b(\Sigma)-\chi(\Sigma)}2
\end{equation}
where $b(\Sigma)$ is the number of boundary components of $\Sigma$ and $\chi(\Sigma)$ is the Euler characteristic of $\Sigma$.
The inclusion-exclusion principle for the Euler characteristic is useful:
\[
\chi(A\cup B) = \chi(A) + \chi(B) - \chi(A\cap B)
\]
It implies that genus is unchanged by capping a boundary component of $\Sigma$ with a disk, and by
removing the interior of a $2$-disk from $\Sigma$.
It also implies that genus is additive with respect to connected sum.
Note that the genus of $\Sigma$ in~\eqref{eq:gen} is a nonnegative integer.
To see that, cap the boundary components of $\Sigma$ with disks and then apply the classification of closed, orientable, connected surfaces.

Let $K\subset S^3$ be a knot.
A \textbf{Seifert surface} $\Sigma$ for $K$ is a compact, orientable, connected surface in $S^3$ whose boundary is $K$.  
Seifert's algorithm~\cite{seifert} produces a Seifert surface for each knot.
By the well-ordering principle, $K$ has a minimal genus Seifert surface.
The \textbf{genus} of $K$ is the genus of a minimal genus Seifert surface for $K$.
Note that minimal genus Seifert surfaces need not be unique---see Rolfsen~\cite[p. 123]{rolfsen} and Roberts~\cite{roberts}.

The following basic facts will be useful.
The proof of Lemma~\ref{lem:sep} is straightforward and left to the reader.

\begin{lemma}\label{lem:sep}
    Let $\alpha$ be a separating arc on a compact, orientable, connected surface $\Sigma$. Let $\Sigma_1$ and $\Sigma_2$ be the two components of the compact surface that results from cutting $\Sigma$ along $\alpha$. Then, $g(\Sigma)=g(\Sigma_1)+g(\Sigma_2)$. 
\end{lemma}

\begin{lemma}\label{lemm:2}
    Let $\alpha$ be an essential simple closed curve in the interior of a compact, orientable, connected surface $\Sigma$. Cut $\Sigma$ along $\alpha$ and cap each of the two resulting boundary circles with a disk to obtain the surface $\Sigma'$. If $\alpha$  is not separating on $\Sigma$, then $g(\Sigma')=g(\Sigma)-1$. If $\alpha$ is separating on $\Sigma$ and $\Sigma$ has at most one boundary component, then each component $\Sigma_i'$ of $\Sigma'$ satisfies $0<g(\Sigma_i')<g(\Sigma)$.
\end{lemma}
\begin{proof}
    Suppose that $\alpha$ is not separating.
    Cut $\Sigma$ along $\alpha$ to obtain the compact surface $\widehat \Sigma$.
    Note that $b\left(\widehat\Sigma\right)=b(\Sigma)+2$ and $\chi\left(\widehat\Sigma\right)=\chi(\Sigma)$.
    Using~\eqref{eq:gen}, we conclude that $g\pa{\widehat\Sigma}=g(\Sigma)-1$.
    As $g(\Sigma')=g\pa{\widehat\Sigma}$, the result follows.
    
    Next, suppose that $\alpha$ is separating.
    Cut $\Sigma$ along $\alpha$ to obtain the compact surface $\widehat{\Sigma}$
    with two components $\widehat{\Sigma}_1$ and $\widehat{\Sigma}_2$.
    Let $\Sigma'_i$ be $\widehat{\Sigma}_i$ capped with a disk
    along the boundary component from $\alpha$.
    Notice that $\Sigma=\Sigma_1'\#\Sigma_2'$.
    As genus is additive with respect to connected sum, we have that: 
    \begin{equation}\label{ineq:gen}
        0\leq g(\Sigma'_i)\leq g(\Sigma) \quad \tn{for $i=1$ and $2$}
    \end{equation}
    Suppose, by way of contradiction, that $g(\Sigma'_i)=0$. Note that $g(\Sigma'_i)=g\left(\widehat\Sigma_i\right)$ and $\widehat\Sigma_i$ is a compact, orientable, connected surface. By the classification of surfaces, $\widehat\Sigma_i$ is a 2-disk with some number of holes. However, since $b(\Sigma)\leq 1$ by hypothesis, $\widehat\Sigma_i$ must be a 2-disk or an annulus, implying that $\alpha$ is inessential in $\Sigma$, a contradiction. Therefore, the inequalities in (\ref{ineq:gen}) are strict. 
    \end{proof}

\section{Proof of Theorem~\ref{thm:main}}\label{sec:proof}
  
Let $\theta=K\cup \alpha$ for some arc $\alpha$ neatly embedded and essential in $\Sigma$. Let $e_1,e_2$, and $e_3$ denote the edges of $\theta$. Without loss of generality, assume that $e_1\cup$ $e_2=K$ so $e_3=\alpha$.  

First, since $K$ is prime, $\theta$ is knotted.

Suppose, seeking a contradiction, that $\theta=\theta'\#_2K'$ for some possibly trivial theta-curve $\theta'$ and some nontrivial knot $K'$. Then, there exists a splitting sphere $S\subset S^3$ that meets $\theta$ at exactly two points $p_1$ and $p_2$ in the interior of an edge of $\theta$. Without loss of generality, assume that $S$ is transverse to $\Sigma$ and $\theta$. Suppose first that $S$ meets either $e_1$ or $e_2$. Note that $\partial(\Sigma\cap S)=\partial\Sigma\cap S$ by transversality. Thus, any arc in $\Sigma\cap S$ must intersect the boundary of $\Sigma$ at two distinct points. Therefore, there is a single arc $\beta\in \Sigma\cap S$ such that $\beta\cap \theta=\{p_1,p_2\}$. Then, $\Sigma\cap S-\beta$ is a finite disjoint union of simple closed curves. Suppose that $\Sigma\cap S-\beta$ is non-empty. We eliminate all simple closed curves in $\Sigma\cap S$ by isotopies of $S$ in $S^3$ fixing $\beta$. First, choose an arbitrary simple closed curve $C\subset\Sigma\cap S$. Then, $C$ divides $S$ into two disks as shown in Figure~\ref{fig:splitting_sphere};
    \begin{figure}[h]
        \centering
        \includegraphics{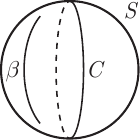}
        \caption{The simple closed curve $C$ bounds a disk in $S$ that does not contain $\beta$.}
        \label{fig:splitting_sphere}
    \end{figure}
    let $D$ be the disk not containing $\beta$. Without loss of generality, assume $C$ is innermost in $D$. Suppose, seeking a contradiction, that $C$ is essential on $\Sigma$. Cut $\Sigma$ along $C$, glue on two parallel copies of $D$, and let $\Sigma'$ denote the resulting surface. If $C$ is not separating on $\Sigma$, then Lemma \ref{lemm:2} implies that $g(\Sigma')<g(\Sigma)$, contradicting minimality of $g(\Sigma)$. If $C$ is separating on $\Sigma$, then one of the components $\Sigma'_i$ of $\Sigma'$ has boundary $K$. Lemma~\ref{lemm:2} implies that $g(\Sigma'_i)<g(\Sigma)$, again contradicting minimality of $g(\Sigma)$. Thus, $C$ is inessential in $\Sigma$ and bounds a disk $D_\Sigma$ in $\Sigma$. So, $D_\Sigma\cup D$ is a sphere and splits $S^3$ into two 3-balls. As $K$ is connected and disjoint from $D_\Sigma\cup D$, $K$ lies in the interior of exactly one of those 3-balls. Use the 3-ball not containing $K$ to push $D$ slightly past $D_\Sigma$ to a parallel copy of $D_\Sigma$, thus removing $C$ from $\Sigma\cap S$. Repeat that step finitely many times to get $\Sigma\cap S=\beta$. Note that $S$ divides $S^3$ into two 3-balls $B_1$ and $B_2$. Any path in $\Sigma$ with one endpoint in $\Int{B_1}$ and the other in $\Int{B_2}$ must meet $S$. Therefore, $\beta$ separates $\Sigma$ into two components $\Sigma\cap B_1$ and $\Sigma\cap B_2$, each of which is a compact, orientable, connected surface with one boundary component. By Lemma \ref{lem:sep}:
    \begin{equation}\label{eq:gencase1}
        g(\Sigma \cap B_1)+g(\Sigma \cap B_2)=g(\Sigma)
    \end{equation}
    Since $K$ is prime, exactly one of $(B_1, B_1\cap K)$ or $(B_2,B_2\cap K)$ is an unknotted ball-arc pair. Without loss of generality, assume $(B_1, B_1\cap K)$ is an unknotted ball-arc pair. 
    Exactly one of the 3-balls $B_1$ or $B_2$ contains $e_3=\alpha$, and the other 3-ball and its intersection with $K$ is a knotted ball-arc pair coming from $K'$ in the supposed 2-connected sum $\theta=\theta'\#_2K'$.
    We will show that $e_3$ must lie in $B_2$, which will be a contradiction since $(B_1, B_1\cap K)$ is an unknotted ball-arc pair.
    As $(B_1, B_1\cap K)$ is an unknotted ball-arc pair,
    $\beta\cup (B_1\cap K)$ bounds a 2-disk $\Delta$ embedded in $B_1$ such that $\Delta$ meets the boundary of $B_1$ in exactly $\beta$. (We note that $\Delta$ is not \emph{neatly} embedded in $B_1$.) If $g(\Sigma \cap B_1)>0$, then by~\eqref{eq:gencase1} we could construct a Seifert surface for $K$ with lower genus than $\Sigma$ by replacing $\Sigma\cap B_1$ with $\Delta$, contradicting minimality of $g(\Sigma)$. Thus, $g(\Sigma \cap B_1)=0$ and  $\Sigma \cap B_1$ is a 2-disk. So, $\Sigma\cap B_1$ contains no neatly embedded essential arcs and $e_3$ must lie in $B_2$, a contradiction.

    Next, suppose that $S$ meets the interior of $e_3$ at $p_1$ and $p_2$.
    As $S$ does not meet $K$, $\Sigma \cap S$ consists of finitely many simple closed curves. Suppose, by way of contradiction, that a component $C$ of $\Sigma\cap S$ is essential in $\Sigma$. Then, $C$ divides $S$ into two 2-disks.
    Let $D$ be one of those disks.
    We show that $D$ contains a component of $\Sigma\cap S$ that is essential in $\Sigma$ and is innermost in $D$.
    If $\Int{D}$ contains no components of $\Sigma\cap S$, then $C$ is the desired component.
    Otherwise, let $C'$ be a component of $\Sigma\cap S$ innermost in $\Int{D}$.
    If $C'$ is inessential in $\Sigma$, then $C'$ bounds a disk $D_S$ in $\Int{D}\subset S$ and a disk $D_\Sigma$ in $\Sigma$.
    Their union is a sphere that splits $S^3$ into two 3-balls. As $K$ is connected and disjoint from $D_S\cup D_{\Sigma}$, $K$ lies in the interior of exactly one of those 3-balls. Use the 3-ball not containing $K$ to push $D_S$ slightly past $D_\Sigma$ to a parallel copy of $D_\Sigma$, removing $C'$ from $\Sigma\cap S$. Repeat that step finitely many times to obtain a component of $\Sigma\cap S$ that is essential in $\Sigma$ and innermost in $D$.
    Without loss of generality, $C$ is a component of $\Sigma\cap S$ that is essential in $\Sigma$ and and bounds a disk $D\subset S$ such that the interior of $D$ does not meet $\Sigma$.
    Now, cut $\Sigma$ along $C$, glue on two copies of $D$, and let $\Sigma'$ denote the resulting surface.
    If $C$ is not separating on $\Sigma$, then Lemma~\ref{lemm:2} implies that $g(\Sigma')<g(\Sigma)$, contradicting minimality of $g(\Sigma)$. If $C$ is separating on $\Sigma$, then one of the two components $\Sigma'_i$ of $\Sigma'$ has boundary $K$. Lemma~\ref{lemm:2} implies that $g(\Sigma'_i)<g(\Sigma)$, again contradicting minimality of $g(\Sigma)$. We conclude that no component of $\Sigma\cap S$ is essential in $\Sigma$. If a component $C$ of $\Sigma\cap S$ is inessential in $\Sigma$ and meets $e_3$, then $C$ meets $\theta$ at two points. 
    Thus, there is exactly one component $C$ of $\Sigma\cap S$---inessential in $\Sigma$---that meets $\theta$ and it does so at exactly $p_1$ and $p_2$. Remove all other components of $\Sigma\cap S$ by isotopies of $S$ as above, so $\Sigma \cap S=C$. Let $B_1$ be the 3-ball bounded by $S$ not containing $K$.
    As $C$ is inessential in $\Sigma$, $C$ bounds a disk $\Delta$ in $\Sigma$.
    By Lemma~\ref{lem:ballpairs}, $(B_1, B_1\cap \theta)$ is an unknotted ball-arc pair. Therefore, $\theta$ is a trivial 2-connected sum, a contradiction.

    Suppose, by way of contradiction, that $\theta=\theta_1\#_3\theta_2$ for two nontrivial theta-curves $\theta_1$ and $\theta_2$. Then, there exists a splitting sphere $S\subset S^3$ that meets $\theta$ at exactly one point in the interior of each edge of $\theta$ and is transverse to $\Sigma$ and $\theta$. Let $p_1$, $p_2$, and $p_3$ be the points in the interiors of the edges $e_1$, $e_2$, and $e_3$ at which $S$ meets $\theta$.
    Note that $\Sigma\cap S$ is a finite disjoint union of arcs and simple closed curves. As above, there is one arc component $\beta$ of $\Sigma \cap S$ that contains $p_1$ and $p_2$.
    We claim that no simple closed curve component of $\Sigma \cap S$ meets $\theta$ at exactly $p_3$.
    To see that, suppose by way of contradiction, that $C$ is a simple closed curve component of $\Sigma \cap S$ and $C$ meets $\theta$ at exactly $p_3$.
    If $C$ is inessential in $\Sigma$, then $C$ meets $e_3$ an even number of times, a contradiction.
    Hence, $C$ must be essential in $\Sigma$.
    That yields a contradiction as above.
    Therefore, $\beta$ contains $p_3$ as well, proving the claim.
    As above, no simple closed curve component of $\Sigma \cap S$ can be essential in $\Sigma$.
    All simple closed curve components of $\Sigma \cap S$ that are inessential in $\Sigma$ can be removed by isotopies of $S$ in $S^3$ fixing $\theta$ as above.
    Thus, we have reduced to the case that $\beta$ is the only component of $\Sigma\cap S$.
    Note that $S$ splits $S^3$ into two 3-balls $B_1$ and $B_2$,
    and $\beta$ separates $\Sigma$ into two surfaces $\Sigma\cap B_1$ and $\Sigma\cap B_2$.
    As $K$ is prime, exactly one of $(B_1, B_1\cap K)$ or $(B_2, B_2\cap K)$ is an unknotted ball-arc pair.
    Without loss of generality, assume $(B_1, B_1\cap K)$ is an unknotted ball-arc pair.
    As $(B_1, B_1\cap K)$ is an unknotted ball-arc pair,
    $\beta\cup (B_1\cap K)$ bounds a 2-disk $\Delta$ embedded in $B_1$ such that $\Delta$ meets the boundary of $B_1$ in exactly $\beta$.
    (Again, note that $\Delta$ is not \emph{neatly} embedded in $B_1$.)
    If $g(\Sigma\cap B_1)>0$, then by~\eqref{eq:gencase1} we could construct a Seifert surface for $K$ with lower genus than $\Sigma$ by replacing $\Sigma\cap B_1$ with $\Delta$, contradicting minimality of $g(\Sigma)$.
    Otherwise, $g(\Sigma\cap B_1)=0$ which means $\Sigma\cap B_1$ is the 2-disk.
    Lemma~\ref{lem:ballpairs} implies that $(B_1, B_1\cap \theta)$ is an unknotted ball-prong pair. Therefore, $\theta$ is not the 3-connected sum of two non-trivial theta-curves, completing the proof. 

\section{Applications and Questions}\label{sec:applications}

We explore the prime theta-curves given by Theorem~\ref{thm:main} and close with a few questions. 
Let $\St$ denote the set of equivalence classes of theta-curves provided by Theorem~\ref{thm:main}.

\begin{coro}
The set $\St$ is countably infinite.
In particular, Theorem~\ref{thm:main} yields a countably infinite collection of pairwise inequivalent prime theta-curves.
\end{coro}

\begin{proof}
There are infinitely many prime knots.
For instance, the collections of torus knots \cite[pp.~49~\&~101]{burde} and hyperbolic knots~\cite[$\S$7.1]{purcell} are both infinite.
Every knot has a minimal genus Seifert surface.
Every Seifert surface of positive genus contains a neatly embedded essential arc (see Figure~\ref{fig:thm_example}).
Therefore, every prime knot appears as a constituent knot of a representative of an element of $\St$.
Since each theta-curve contains three constituent knots, $\St$ is infinite.

For countability, recall that each compact surface and compact graph piecewise linearly embedded in $S^3$ is a union of finitely many simplices in a subdivision of the standard triangulation of $S^3$ as the boundary of the $4$-simplex.
There are only countably many such subsets.
\end{proof}

The previous corollary used only the existence of one neatly embedded essential arc in each Seifert surface of a prime knot.
In fact, there is a rich collection of such arcs.
Let $\Sigma$ be a compact, orientable, connected surface of positive genus with one boundary component $K$.
The mapping class group $\mcg{\Sigma}$ of $\Sigma$ acts on the set of neatly embedded essential arcs in $\Sigma$.
The resulting arcs union $K$ can produce many different prime theta-curves.
If $g(\Sigma)=1$, then every neatly embedded essential arc in $\Sigma$ is nonseparating by Lemma~\ref{lem:sep}.
If $g(\Sigma)>1$, then $\Sigma$ contains both nonseparating and separating neatly embedded essential arcs.
If such an arc $\alpha$ separates $\Sigma$, then cutting $\Sigma$ along $\alpha$ yields two compact, orientable, connected surfaces whose genera $g_1$ and $g_2$ sum to $g(\Sigma)$ by Lemma~\ref{lem:sep}.
In that case, let the \textbf{type} of $\alpha$ be the set $\cpa{g_1,g_2}$. 
Notice that $\mcg{\Sigma}$ acts transitively on the set of nonseparating neatly embedded essential arcs on $\Sigma$.
To see that, consider two such arcs $\alpha$ and $\alpha'$ in separate copies of $\Sigma$.
Cut each surface along its respective arc.
By the classification of surfaces, the resulting surfaces are homeomorphic.
As $\Sigma$ is orientable, we may carefully glue each of those two surfaces back together to obtain a self-homeomorphism of $\Sigma$ carrying $\alpha$ to $\alpha'$, as desired.
Similarly, $\mcg{\Sigma}$ acts transitively on the set of separating neatly embedded essential arcs on $\Sigma$ of each type.

We present another way to construct neatly embedded essential arcs in a Seifert surface. This method will be used below.

\begin{lemma}\label{lem:knottoessarc}
Let $K\subset S^3$ be a nontrivial knot and let $\Sigma\subset S^3$ be a Seifert surface for $K$.
Let $J\subset\Int{\Sigma}$ be a nontrivial knot that is not boundary parallel---meaning $J$ and $K$ do not bound an annulus in $\Sigma$.
Isotop $J$ on $\Sigma$ in any way to $J'$ so that exactly an arc $e_1$ of $J'$ lies in $K$.
Define the arcs $e_2=K-\Int{e_1}$ and $e_3=J'-\Int{e_1}$.
Then, $e_3$ is neatly embedded and essential in $\Sigma$. Furthermore, the theta-curve $K\cup e_3$ has $K$ and $J'$ as constituent knots.
\end{lemma}

Note that the third consituent knot of $K\cup e_3$ need not be $K\#J'$.

\begin{proof}
Suppose, by way of contradiction, that $e_3$ is inessential in $\Sigma$.
If $e_1\cup e_3$ bounds a disk in $\Sigma$, then $J'$ and, hence $J$, is unknotted, a contradiction.
If $e_2\cup e_3$ bounds a disk in $\Sigma$, then $J$ is boundary parallel, a contradiction. 
\end{proof}

The first author and Nieman~\cite{calcut2025} classified
all \textbf{torus theta-curves}---theta-curves that lie on a standard torus $T$ in $S^3$.
Part of that classification involved proving certain torus theta-curves are prime.
Let $t(p,q)$ be a nontrivial torus knot,
so $p$ and $q$ are relatively prime integers and $\card{p},\card{q}\geq2$.
Let $\theta(p,q)$ be the theta-curve obtained by adding to $t(p,q)$ a simplest
possible arc on $T$ that is essential in the annulus obtained by cutting $T$ along $t(p,q)$.
Theorem~3.1 of~\cite{calcut2025} proves that $\theta(p,q)$ is prime.
Our Theorem~\ref{thm:main} implies that theorem as follows.
For clarity, we exhibit our argument in the specific case $\theta(3,4)$ as in Figure~\ref{fig:torus}.
\begin{figure}[ht!]
\centering
    \begin{subfigure}[b]{0.5\textwidth}
        \centering
        \includegraphics[]{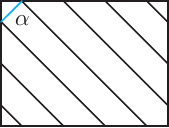}
        \caption{Torus knot $t(3,4)$ and arc $\alpha$ on the torus $T$.}
    \end{subfigure}
 
    \begin{subfigure}[b]{0.3\textwidth}
        \centering
        \includegraphics[]{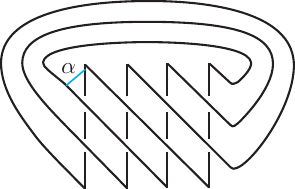}
        \caption{Projection of $t(3,4)$ together with arc $\alpha$.}
    \end{subfigure}
  \hspace{2cm}
    \begin{subfigure}[b]{0.3\textwidth}
        \centering
        \includegraphics[]{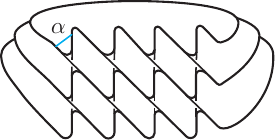}
        \caption{Seifert surface of $t(3,4)$ with arc $\alpha$.}
    \end{subfigure}
   
    \caption{Algorithm to construct a minimal genus Seifert surface for a torus knot.}
    \label{fig:torus}
\end{figure}
The general case follows similarly.
Figure~\ref{fig:torus} shows a standard way to construct a Seifert surface
of $t(p,q)$ of genus $(p-1)(q-1)/2$.
That surface has minimal genus for $t(p,q)$ by~\cite[Cor.~4.12,~\&~Ex.~9.15]{burde}.
As $\alpha$ lies in that minimal genus Seifert surface for $t(p,q)$ and is nonseparating (hence essential),
Theorem~\ref{thm:main} implies that $\theta(p,q)$ is prime, as desired.

Our results provide the following new example.
The figure eight knot $J$ lies in the interior of a minimal genus
Seifert surface $\Sigma$ for $t(3,4)$ as shown in Figure~\ref{fig:figure8}.
\begin{figure}
    \centering
    \includegraphics{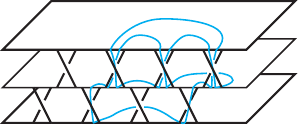}
    \caption{The figure eight knot $J$ in a minimal genus Seifert surface for $t(3,4)$.}
    \label{fig:figure8}
\end{figure}
Note that $J$ is not boundary parallel in $\Sigma$ since the figure eight knot is nontrivial and
amphichiral, whereas all nontrivial torus knots are not amphichiral~\cite[Thm.~3.39]{burde}.
Lemma~\ref{lem:knottoessarc} and Theorem~\ref{thm:main} imply that there is a prime
theta-curve $\theta$ with constituent knots $t(3,4)$ and the figure eight knot.
That means $\theta$ cannot be a torus theta-curve. 

Not every prime theta-curve represents an element of $\St$. For example, Kinoshita's~\cite{kinoshita} theta-curve is prime---see the discussion in~\cite{taylor_ozawa}---yet all its constituent knots are trivial. That raises the question, do there exist prime theta-curves with at least one prime constituent knot that do not represent an element of $\St$? 

Our focus has been on adding arcs to prime knots. The first author and Metcalf-Burton determined which theta-curves are prime when arcs are added to the unknot~\cite{CalcutMB}. How can one add arcs to composite knots to create prime theta-curves?

Let $\Sigma$ be a Seifert surface for a given prime knot $K$. Let $\alpha$ denote any neatly embedded essential arc in $\Sigma$.
Theorem~\ref{thm:main} gives a necessary condition for $\Sigma$ to be of minimal genus for $K$:
$$\Sigma \tn{ minimal genus for } K \quad\implies \quad K\cup \alpha \tn{ is prime for all }\alpha$$
Does the converse hold? 

\bibliography{References.bib}
\bibliographystyle{plainurl}
\end{document}